\documentclass[12pt]{amsart}
\usepackage{amssymb,amsmath,amsthm}
\usepackage{enumerate, url}
\newtheorem{theorem}{Theorem}
\newtheorem{lemma}[theorem]{Lemma}

\theoremstyle{remark}

\newtheorem*{remark}{Remark}

\numberwithin{theorem}{section} \numberwithin{equation}{section}

\newcommand{\calO}{\mathcal{O}}

\newcommand{\mfp}{\mathfrak{p}}
\newcommand{\mfb}{\mathfrak{b}}
\newcommand{\mfc}{\mathfrak{c}}

\newcommand{\R}{\mathbb{R}}
\newcommand{\C}{\mathbb{C}}
\newcommand{\F}{\mathbb{F}}
\newcommand{\Cl}{{\text {\rm Cl}}}

\newcommand{\Aut}{{\text {\rm Aut}}}
\newcommand{\Q}{\mathbb{Q}}

\newcommand{\Z}{\mathbb{Z}}

\newcommand{\SL}{{\text {\rm SL}}}
\newcommand{\GL}{{\text {\rm GL}}}
\newcommand{\textmod}{{\text {\rm mod}}}

\newcommand{\sgn}{\operatorname{sgn}}

\newcommand{\Gal}{{\text {\rm Gal}}}

\newcommand{\Disc}{\textnormal{Disc}}
\newcommand{\Stab}{\textnormal{Stab}}

\newcommand{\A}{\mathbb{A}}

\begin{document}
\title[Shintani's zeta function is not a finite sum of Euler products]
{Shintani's zeta function is not a finite sum of Euler products}

\author{Frank Thorne}
\address{Department of Mathematics, University of South Carolina, 1523 Greene Street, Columbia, SC 29201}
\email{thorne@math.sc.edu}

\begin{abstract}
We prove that the Shintani zeta function associated to the space of binary cubic forms cannot be written
as a finite sum of Euler products. Our proof also extends to several closely related Dirichlet series. This answers a
question of Wright \cite{W} in the negative.

\end{abstract}

\maketitle
\section{Introduction}
In this paper, we will prove that {\itshape Shintani's zeta function is not a finite sum of Euler products}.
We will also prove the same for the Dirichlet series counting cubic fields.

Our work is motivated by a beautiful paper of Wright \cite{W}. To illustrate Wright's work, we recall a
classical example, namely the Dirichlet series associated to fundamental discriminants. We
have
\begin{equation}
\sum_{D > 0} D^{-s} = \frac{1}{2} \bigg[ \big(1 - 2^{-s} + 2 \cdot 4^{-s}\big) \frac{\zeta(s)}{\zeta(2s)} +
\big( 1 - 4^{-s} \big) \frac{L(s, \chi_4)}{L(2s, \chi_4)}\bigg],
\end{equation}
as well as a similar formula for negative discriminants. These formulas may look a bit messy.
However, if one combines positive and negative discriminants, one has the beautiful formulas
\begin{equation}\label{eqn_beautiful}
\sum |D|^{-s} = \Big( 1 - 2^{-s} + 2 \cdot 4^{-s}\Big) \frac{\zeta(s)}{\zeta(2s)} =
\prod_p \bigg( \frac{1}{2} \sum_{[K_v : \Q_p] \leq 2} |\Disc(K_v)|_p^s \bigg),
\end{equation}
and
\begin{equation}\label{eqn_beautiful2}
\sum \sgn(D) |D|^{-s} = \Big( 1 - 4^{-s}\Big) \frac{L(s, \chi_4)}{L(2s, \chi_4)}.
\end{equation}
These are special cases of much more general results. Wright obtains similar formulas for quadratic
extensions of any global field $k$ of characteristic not equal to 2, which are in turn the case $n = 2$ of formulas
for Dirichlet series parameterizing the elements of $k^{\times}/(k^{\times})^n$. He proves his results by considering 
twists of the Iwasawa-Tate zeta function
\begin{equation}\label{eqn_it}
\zeta^{(n)}(\omega, \Phi) = \int_{\A^{\times} / k^{\times}} \omega(t) \sum_{x \in k^{\times}} \Phi(t^n x) |d^{\times} t|_{\A}.
\end{equation}
We will not explain Wright's notation here, but suffice it to say that the case $n = 1$ is the zeta function of
Tate's thesis \cite{tate}. This zeta function may be viewed as the zeta function associated to the affine line, viewed as a
prehomogeneous vector space of degree 1, on which $\GL(1)$ acts by $\phi(t) x = t^n x$. 
Generalizing \eqref{eqn_beautiful} and \eqref{eqn_beautiful2},
Wright proves 
that all of these zeta functions can be written as finite linear combinations of Euler products.
He further remarks that ``an analogue of [these formulas] for the space of binary cubic
forms is currently unknown; although, such a formula would be immensely interesting.''

Wright is referring to the {\itshape Shintani zeta function} associated to the space of binary cubic forms.
This zeta function was introduced by Shintani \cite{S} and further studied by Datskovsky and Wright \cite{DW1, DW2, DW3} and 
many others. This zeta function is defined as follows:
The lattice of {\itshape integral binary cubic forms} is defined by
\begin{equation}\label{def_vz}
V_{\Z} := \{ a u^3 + b u^2 v + c u v^2 + d v^3 \ : a, b, c, d \in \Z \},
\end{equation}
and the {\itshape discriminant} of such a form is given by the usual equation
\begin{equation}\label{eqn_disc_formula}
\Disc(f) = b^2 c^2 - 4 a c^3 - 4 b^3 d - 27 a^2 d^2 + 18 abcd.
\end{equation}
There is a natural action of $\GL_2(\Z)$ (and also of $\SL_2(\Z)$) on $V_{\Z}$, given by
\begin{equation}
(\gamma \cdot f)(u, v) = \frac{1}{\det \gamma} f((u, v) \cdot \gamma).
\end{equation}
The {\itshape Shintani zeta functions} are given by the Dirichlet series
\begin{equation}\label{eqn_def_shintani}
\xi^{\pm}(s) := \sum_{\substack{x \in \SL_2(\Z) \backslash V_{\Z} \\ \pm \Disc(x) > 0}} \frac{1}{|\Stab(x)|} |\Disc(x)|^{-s},
\end{equation}
and Shintani proved that they have analytic continuation to $\C$ and satisfy a functional equation.
These zeta functions are interesting for a variety of reasons; see, e.g., \cite{TT} and \cite{DW3} for
applications to counting cubic extensions of number fields. Also note that Datskovsky and Wright's
work yields an adelic formulation of Shintani's zeta function, similar to \eqref{eqn_it}.

Some further motivation for Wright's question is provided by work of Ibukiyama and Saito \cite{IS}.
They consider the zeta functions associated to the prehomogeneous vector spaces of $n \times n$
symmetric matrices, for $n > 3$ odd. In particular they prove, among many other interesting results, 
explicit formulas for these zeta functions as sums of two products of the Riemann zeta function. 

With these results in mind, one might hope to prove similar such formulas 
for the zeta functions \eqref{eqn_def_shintani}. In this note we answer Wright's
question in the negative, and prove that no such formulas exist.
\begin{theorem}\label{thm_main}
Neither of the Shintani zeta functions $\xi^{\pm}(s)$ defined in \eqref{eqn_def_shintani} admits a
representation as a finite sum of Euler products.
\end{theorem}
In other words, if we write $\xi^{\pm}(s) = \sum_n a(n) n^{-s}$, we cannot write $a(n) = \sum_{i = 1}^k c_i b_i(n)$
for real numbers $c_i$ and multiplicative functions $b_i(n)$.

It is interesting to note that these zeta functions do have representations as infinite sums of Euler
products. Of course, there is a tautological such representation, writing
$n^{-s} = (1 + n^{-s}) - 1$
and thus regarding each term in \eqref{eqn_def_shintani} as a difference of Euler products.
However, Datskovsky and Wright \cite{DW2} proved the much more interesting formula
\begin{equation}\label{eqn_dw2}
\xi^{\pm}(s) = \zeta(4s) \zeta(6s - 1) \sum_k \frac{2}{o(k)} |\Disc(k)|^{-s} \frac{R_k(2s)}{R_k(4s)},
\end{equation}
where the sum is over all number fields of degree 1, 2, or 3 (up to isomorphism) of the correct sign; $o(k)$ is equal to
6, 2, 1, or 3 when $k$ is trivial, quadratic, cubic and non-Galois, or cubic and cyclic respectively;
and $R_k(s)$ is equal
to $\zeta(s)^3, \zeta(s) \zeta_k(s), \zeta_k(s), \zeta_k(s)$ respectively.

The proof of Theorem \ref{thm_main} is not difficult, and it illustrates an application
of recent work of Cohen and Morra (\cite{CM}; see also Morra's thesis \cite{M} for a longer version with more examples). 
For a fundamental discriminant $D$, 
they prove explicit formula for the Dirichlet series
\begin{equation}
\Phi_D(s) := \sum_n a(Dn^2) n^{-s},
\end{equation}
where $a(Dn^2)$ is the number of cubic fields of discriminant $D n^2$. More generally, they prove
a formula for the Dirichlet series counting cubic extensions of number fields 
$K/k$, such that the normal
closure of $K$ contains a fixed quadratic subextension $K_2/k$. The formula is rather complicated, and it
involves sums over characters of the 3-part of ray class groups associated to $\Q(\sqrt{D}, \sqrt{-3})$
(or more generally to $K_2(\sqrt{-3})$). However, these formulas take much simpler forms when we assume $k = \Q$, and when
we further restrict
to $D$ for which we can control these class groups.

We will apply the following special case of their result:

\begin{theorem}[Cohen-Morra \cite{CM}]\label{thm_cm}
If $D < 0$, $D \equiv 3 \ (\textmod \ 9)$, and $3 \nmid h(D)$, then we have
\begin{equation}\label{eqn_cm_im}
\sum_n a(Dn^2) n^{-s} = - \frac{1}{2} + \frac{1}{2} \bigg(1 + \frac{2}{3^s} \bigg)
\prod_{ \left( \frac{ -3D}{p} \right) = 1 } \bigg(1 + \frac{2}{p^s} \bigg).
\end{equation}
Furthermore, if $D > 0$, $D \equiv 3 \ (\textmod \ 9)$, 
and $3 \nmid h(-D/3)$, then there exists a cubic field $E$ of discriminant $-27 D$ such that
\begin{multline}\label{eqn_cm_re}
\sum_n a(Dn^2) n^{-s} =  - \frac{1}{2} + \frac{1}{6} \bigg(1 + \frac{2}{3^s} \bigg)
\prod_{ \left( \frac{ -3D}{p} \right) = 1 } \bigg(1 + \frac{2}{p^s} \bigg)
\\ 
+ \frac{1}{3} \bigg(1 - \frac{1}{3^s} \bigg)
\prod_{ \left( \frac{ -3D}{p} \right) = 1 } \bigg(1 + \frac{\omega_D(p)}{p^s} \bigg), \ \ \ \ \  \ \ \
\end{multline}
where $\omega_D(p) = 2$
if $p$ splits completely in $E$, and $\omega_D(p) = -1$ otherwise. 
\end{theorem}
\begin{remark}
In fact, $p$ splits completely in $E$ if and only if it does so in its Galois closure
$E(\sqrt{-3D})$, and this version of the condition will be more convenient in our proof.
\end{remark}

The equation \eqref{eqn_cm_im} is Corollary 7.9 of \cite{CM}. The analogue \eqref{eqn_cm_re}
does not appear in \cite{CM, M}, so in Section \ref{sec_cm} we describe how this formula follows from Cohen and Morra's
work. Moreover, in forthcoming work with Cohen \cite{CT} we will generalize \eqref{eqn_cm_re} to any $D$, positive or negative,
and we will also prove an analogous formula for {\itshape quartic} fields having fixed {\itshape cubic} resolvent.

We can now summarize our proof of Theorem \ref{thm_main}, in the
easier negative discriminant case. First observe that $\sum_n a(n) n^{-s} := \xi^-(s) - \zeta(s)$
cannot be an Euler product: By \eqref{eqn_cm_im}, we may find a negative 
$D \equiv 3 \ (\textmod \ 9)$ and many primes $p$ for which $a(-D) = 0$
and $a(-D p^2) > 0$.

To prove that
that $\xi^-(s) - \zeta(s)$ is not a sum of $k$ Euler products, for any $k > 1$, we choose appropriate
discriminants $D_i$ and primes $p_i$ such that
$a(-D_i p_j^2) > 0$ if and only if $i = j$. Then $a(n)$ ``takes too many parameters to describe,''
and we will formalize this using an elementary linear algebra argument.

The argument for the positive discriminant case is mildly more complicated, but essentially the same.
\\
\\
{\bf Closing remarks.} Although we cheerfully admit that a positive result would be more interesting than the present paper,
we submit that our negative result is interesting as well.

One motivation for our work comes from our previous paper \cite{T}, where we studied the Shintani zeta functions analytically,
obtaining (limited) results on the location of
the zeroes. These zeta functions essentially fit into the Selberg class, which naturally leads to a host of questions.
For example, do most of the zeroes of the Shintani zeta functions
lie on the critical line? Work of Bombieri and Hejhal \cite{BH} establishes that this is true for certain finite sums of $L$-functions,
conditional on standard hypotheses for these $L$-functions (including GRH).
As Shintani's zeta functions lie outside the scope of \cite{BH},
even a good conjecture would be extremely interesting.

Our work was also motivated by the desire to incorporate recent developments in the study of cubic fields into the study
of the Shintani zeta function. In addition to the work of Cohen and Morra, we point out the work of 
Bhargava, Shankar, and Tsimerman \cite{BST}, Hough \cite{H}, and Zhao \cite{Z} for further approaches to related questions.
We write this paper in optimism that further connections will be found among the various approaches to the subject.

\section{Acknowledgments}
I would like to thank Takashi Taniguchi, Simon Rubinstein-Salzedo, and the referee for useful comments, and 
Yasuo Ohno
for pointing out the reference \cite{IS} to me, which indirectly motivated this paper.
Finally, I would like to thank Henri Cohen and Anna Morra for answering my many questions about their interesting work.

\section{Proof of Theorem \ref{thm_cm}}\label{sec_cm}
In this section, we briefly describe the work of Cohen and Morra \cite{CM, M}, and discuss how 
\eqref{eqn_cm_re} follows from their more general results.\footnote{Our specific references are to the published paper 
\cite{CM}, but we also recommend \cite{M} for its enlightening additional explanation and examples.} 

In a followup paper with Cohen \cite{CT}, we will extend Theorem \ref{thm_cm} and the proof given here to cover the case where $D$
is any fundamental discriminant. The original work of Cohen and Morra is more general still, enumerating relative cubic extensions
of any number field. The new ingredient here (and in \cite{CT}) is that the Cohen-Morra formulas are somewhat abstract, and it is not obvious that they yield
explicit formulas like \eqref{eqn_cm_re}.

Suppose then that $K/\Q$ is a cubic field of discriminant $D n^2$, where $D \not \in \{ 1, -3 \}$ is a fundamental discriminant, and 
let $N$ be the Galois closure of $K$. Then $N(\sqrt{-3})$ is a cyclic cubic extension of $L := \Q(\sqrt{D}, \sqrt{-3})$,
and Kummer theory implies that $N(\sqrt{-3}) = L(\alpha^{1/3})$ for some $\alpha \in L$. We write (following \cite[Remark 2.2]{CM})
\begin{equation}
\Gal(L/\Q) = \{1, \ \tau, \ \tau_2, \ \tau \tau_2 \},
\end{equation}
where $\tau, \ \tau_2, \ \tau \tau_2$ fix $\sqrt{D}, \ \sqrt{-3}, \ \sqrt{-3D}$ respectively.

The starting point of \cite{CM}
is a correspondence between such fields $K$ and such elements $\alpha$.
In particular, isomorphism classes of such $K$ are in bijection with equivalence classes of elements $1 \neq \overline{\alpha} \in L^{\times}/(L^{\times})^3$,
with $\alpha$ identified with its inverse,
such that $\alpha \tau'(\alpha) \in (L^{\times})^3$ for $\tau' \in \{ \tau, \tau_2 \}$. 
We say (as in \cite[Definition 2.3]{CM}) that $\overline{\alpha} \in (L^{\times}/(L^{\times})^3)[T]$, where
$T \subseteq \F_3[\Gal(L/\Q)]$ is defined by $T = \{ \tau + 1, \tau_2 + 1 \}$, and the notation $[T]$ means
that $\overline{\alpha}$ is annihilated by $T$. This bijection opens the door to Cohen-Morra's
further study of such $\overline{\alpha}$ in terms of the ideals of $L$.

Assume that $D \equiv 3 \ (\textmod \ 9)$, placing us in case (5) of \cite[p. 464]{CM}. The 
formula \eqref{eqn_cm_im} is Corollary 7.9 of \cite{CM}, and so we assume that $D > 0$, for which
the conditions for \eqref{eqn_cm_re} are discussed in
Remark 7.8. Define $G_{\mfb}$, as in \cite[Theorem 6.1]{CM}, to be
$\big(\Cl_{\mfb}(L)/\Cl_{\mfb}(L)^3\big)[T],$ where $\Cl_{\mfb}(L)$ is the ray class group of an ideal $\mfb$.
By Remark 7.8, $G_{\mfb}$ is of order 1 (for all $\mfb$ considered in Theorem 6.1), except when $\mfb = (3\sqrt{-3})$, 
one case where it has order 3. 

In this setting, the main theorem (Theorem 6.1 of \cite{CM}) reduces to a formula of the shape
\begin{equation}\label{eqn_CM61}
\sum_n a(D n^2) n^{-s} = -\frac{1}{2} + \frac{1}{2} \sum_{\mfb \in \mathcal{B}} A_{\mfb}(s)
\sum_{\chi \in \widehat{G_{\mfb}}} \omega_{\chi}(3) \prod_{ \big( \frac{-3D}{p} \big) = 1}
\bigg(1 + \frac{ \omega_{\chi}(p)}{p^s} \bigg),
\end{equation} 
where for $p \neq 3$ we have
$$\omega_{\chi}(p)=\begin{cases}
2&\text{\quad if $\chi(p\mfc)=\chi(p\tau(\mfc))$}\;,\\
-1&\text{\quad if $\chi(p\mfc)\ne\chi(p\tau(\mfc))$}\;.\end{cases}$$
We will define $\mathfrak{c}$ later, but we leave the definitions of the other quantities to \cite{CM, M} (see also
our forthcoming followup \cite{CT}). Most of the computations needed to verify \eqref{eqn_cm_re} are either
completely straightforward or carried out in \cite{CM, M}. In particular, the contributions of the trivial characters
are computed there. The one remaining step requiring a substantial argument is
to prove that for each nontrivial character $\chi$ of $G_{(3\sqrt{-3})}$, 
we have $\omega_{\chi}(p) = \omega_D(p)$.
This involves a bit of class field theory, and we give the proof here.

Define\footnote{We have followed the notation of \cite{CM} where practical, but the notations
$G'_{\mfb}, \ G''_{\mfb}, \ E, \ E_1, \omega_D(p)$, and some of the notation appearing in \eqref{eqn_CM61} are used for the first time here and do not
appear in \cite{CM}.} $G'_{\mfb} := \Cl_{\mfb}(L)/\Cl_{\mfb}(L)^3$, so that $G'_{\mfb}$ is
a 3-torsion group containing $G_{\mfb}$. We have a canonical decomposition of $G'_{\mfb}$ into four
eigenspaces for the actions of $\tau$ and $\tau_2$, and we write
\begin{equation}\label{eqn_nci}
G'_{\mfb} \simeq G_{\mfb} \times G''_{\mfb}
\end{equation}
where $G''_{\mfb}$ is the direct sum of the three eigenspaces other than $G_{\mfb}$. Note that $G''_{\mfb}$
will contain the classes of all principal ideals generated by rational integers coprime to $3$; any such class
in the kernel of $T$ will necessarily be in $\Cl_{\mfb}(L)^3$. We may thereby write
$G_{\mfb} \simeq \Cl_{\mfb}(L)/H$, where $H$ is of index $3$ and contains the classes of all principal ideals
coprime to $3$.

By class field theory, there is a unique abelian extension $E_1/L$ for which the Artin map induces an isomorphism
$G_{\mfb} \simeq \Gal(E_1/L)$. It must be cyclic cubic, as $G_{\mfb}$ is, and the uniqueness forces $E_1/L$ to be 
Galois over $\Q$, as the group $G_{\mfb}$ is preserved by $\tau$ and $\tau_2$ and hence all of $\Gal(L/\Q)$.
We have $\Gal(E_1/\Q) \simeq S_3 \times C_2$: $\tau$ and $\tau_2 \in \Gal(L/\Q)$ both act nontrivially on $G_{\mfb}$,
and under the Artin map this implies that $\tau, \tau_2$ both act  nontrivially on $\Gal(E_1/L)$ by conjugation. This forces
$\tau \tau_2$ 
to commute\footnote{More properly speaking, we choose an arbitrary lift of $\tau \tau_2$ to $\Gal(E_1/\Q)$ and denote it also by
$\tau \tau_2$; it is this lift which commutes with $\Gal(E_1/L)$.}
with $\Gal(E_1/L)$. As $\tau \tau_2$ fixes $\Q(\sqrt{-3D})$, this implies that $E_1$ contains
a cubic extension $E/\Q$ with quadratic
resolvent $\Q(\sqrt{-3D})$, which is unique up to isomorphism. Any prime $p$
which splits in $\Q(\sqrt{-3D})$ must either be inert or totally split in $E$.

For each prime $p$ with $\big( \frac{-3D}{p} \big) = 1$, write $p \calO_L = \mfc \tau'(\mfc)$, where $\tau'$ is the element
$\tau$ or $\tau_2$ of $\Gal(L/\Q)$ described previously. Here $\mfc$ is either prime or a product of two primes $\mfp \tau \tau_2(\mfp)$ 
depending on whether
$\big( \frac{D}{p} \big)$ is $-1$ or $1$.
For each such $p$,
$\omega_{\chi}(p) = 2$ if $\chi(p \mfc) = \chi(p \tau'(\mfc))$, and $\omega_{\chi}(p) = -1$ if  $\chi(p \mfc) \neq \chi(p \tau'(\mfc))$.
With the isomorphism $G_{\mfb} \simeq \Cl_{\mfb}(L)/H$, we have $\chi((p)) = 1$, so that $\chi(\mfc)$ is well defined
and equal to $\chi(p \mfc)$ (and similarly for $\chi(\tau'(\mfc))$).
Therefore, the computation
$\chi(\mfc) \chi(\tau'(\mfc)) = \chi((p)) = 1$ implies
that $\chi(\mfc)$ and $\chi(\tau'(\mfc))$ are complex conjugates, and thus
$\omega_{\chi}(p)$ is $2$ or $-1$ depending on whether $\chi(\mfc) = 1$ or not.

We claim that that $\chi(\mfc) = 1$ if and only if $p$ splits completely in $E$. Suppose first that
$\mfc$ is prime in $\calO_L$. Then $\chi(\mfc) = 1$ if and only if $\mfc$ splits completely in
$E_1/L$, in which case $(p)$ splits into six ideals in $E_1$, which happens
if and only if $p$ splits completely
in $E$.

Suppose now that $\mfc = \mfp \tau \tau_2(\mfp)$ in $L$. We claim that $\chi(\mfc) = 1$ if and only if
$\chi(\mfp) = 1$; this follows 
as $\mfp$ and $\tau \tau_2(\mfp)$ have the same Frobenius element in $E_1/L$ (which in turn is true because
they represent the same element of $G_{\mfb}$), and $\chi(\mfc) = \chi(\mfp \tau \tau_2(\mfp)) = 1$.
Now $\chi(\mfp) = 1$ if and only if $\mfp$ splits completely in $E_1/L$, in which case $(p)$ splits into twelve ideals
in $E_1$; for this it is necessary and sufficient that $p$ split completely in $E$.

This proves that $\omega_{\chi}(p) = \omega_D(p)$, as desired, provided that we verify that the discriminant of $E$ is as claimed. As its quadratic resolvent is
$\Q(\sqrt{-3D})$, we observe that  $\Disc(E) = r^2 (-D/3)$ for some integer $r$ divisible only by $3$ and prime divisors of $D$.
No prime $\ell > 3$ can divide $r$, because $\ell^3$ cannot divide the discriminant of any cubic field. Similarly $2$
cannot divide $r$, as if $2 | D$ then $4 | D$, but 16 cannot divide the discriminant of a cubic field. Therefore $r$ must
be a power of 3. We cannot have $r = 1$, as $E(\sqrt{-3D})/\Q(\sqrt{-3D})$ would be an unramified cubic extension,
but $h(-D/3) = 1$. $r$ cannot be $3$, nor can we have $27 | r$, as the $3$-adic valuation of a cubic field discriminant
is never $2$ or larger than $5$. By process of elimination, $r = 9$.

\begin{remark}
Belabas \cite{bel} has computed a table of all cubic fields $K$ with $|\Disc(K)| < 10^6$, and
we used his table and PARI/GP \cite{pari} to double-check \eqref{eqn_cm_re}. 
\end{remark}

\section{Proof of Theorem \ref{thm_main}}
We will require the following lemma, which we extract from a generalization of the Davenport-Heilbronn theorem \cite{DH}.
\begin{lemma}\label{lem_dh}
Given any residue class $a \ (\textmod \ 36m)$ with $a \equiv 21 \ (\textmod \ 36)$ or $a \equiv 5 \ (\textmod \ 12)$, and $(6a, m) = 1$, 
there exist infinitely many negative fundamental discriminants $n \equiv a \ (\textmod \ 36m)$ which are not discriminants of cubic fields.
\end{lemma}
\begin{proof}
This follows from quantitative versions of the Davenport-Heilbronn theorem in arithmetic progressions. See \cite{TT}
for sharp quantitative results; alternatively, the methods of \cite{HN} may be adapted to give an easier proof.

In particular, the results stated in \cite{TT} imply that the average number of cubic fields $K$ per negative fundamental discriminant
is equal to $\frac{1}{2}$, and that this average is the same when restricted to any arithmetic progression as above
containing fundamental
discriminants. The lemma follows from the fact that this average is less than 1.
\end{proof}

\begin{proof}[Proof of Theorem \ref{thm_main}]
For simplicity we begin by proving the theorem for the
negative discriminant zeta function $\xi^-(s)$. The basic idea is the same for all cases.

For reasons that will become apparent later, we begin by subtracting the Riemann zeta function $\zeta(s)$.
Suppose that the function $\xi^-(s) - \zeta(s)$ has a representation
\begin{equation}\label{eqn_def_an}
\xi^-(s) - \zeta(s) = \sum_n a(n) n^{-s} =  
\sum_n \bigg( \sum_{i = 1}^k c_i b_i(n) \bigg) n^{-s},
\end{equation}
where the $c_i$ are nonzero real numbers and the $b_i(n)$ are multiplicative functions. We will obtain a contradiction,
implying Theorem \ref{thm_main}.

Begin by using Lemma \ref{lem_dh} to choose $k$ odd negative fundamental discriminants $n_i \equiv 21 \ (\textmod \ 36)$,
coprime to each other apart from the common factor of 3, for which there are no cubic fields of discriminant $n_i$.

Now, for $1 \leq i \leq k$, choose primes $p_i \equiv 1 \ (\textmod \ 3)$ for which 
$\big( \frac{-3 n_i}{p_j} \big) = \big( \frac{n_i}{p_j} \big) = -1$
for $i \neq j$, and $\big( \frac{-3 n_i}{p_i} \big) = \big( \frac{n_i}{p_i} \big) = 1$ for each $i$.
(These conditions amount to arithmetic progressions mod $\prod n_i$.) We will count the number of
cubic rings with discriminant $n_i$, $n_i p_i^2$, and $n_i p_j^2$, for each $i$ and $j$.

For each $n_i$, there is exactly one quadratic field of discriminant $n_i$. There is exactly one nonmaximal
reducible ring of discriminant $n_i p_j^2$ for $j \neq i$, and
three such rings of discriminant $n_i p_i^2$; these nonmaximal rings are counted by \eqref{eqn_dw2}.

By hypothesis, there are no cubic fields of discriminant $n_i$. By \eqref{eqn_cm_im}, there are no
cubic fields (or nonmaximal rings) of discriminant $n_i p_j^2$ for $j \neq i$, and there is one cubic field
of discriminant $n_i p_i^2$.

As the Shintani zeta function counts noncyclic cubic fields with weight 2, and quadratic rings with weight 1,
we conclude that $a(n_i p_j^2) - a(n_i)$ is equal to zero if $i \neq j$, and is positive if $i = j$. This fact 
is enough to contradict \eqref{eqn_def_an}. To see this, 
let $B$ be
the $k \times k$ matrix with $(r, s)$-entry $c_s b_s(n_r)$, where $c_s$ and $b_s(n)$ are as in \eqref{eqn_def_an},
and let $C_i$ be the vector (e.g., the $n \times 1$ matrix) whose $r$th row is equal to $b_r(p_i^2) - b_r(1)$.
Then $B C_i$ is the vector whose $r$th row is equal to 
\begin{equation}
\sum_{\ell = 1}^k c_{\ell} b_{\ell}(n_r) \big(
b_{\ell}(p_i^2) - b_\ell(1) \big) =
\sum_{\ell = 1}^k \Big( c_{\ell} b_{\ell} (n_r p_i^2) - c_{\ell} b_{\ell}(n_r) \Big)
= a(n_r p_i^2) - a(n_r).
\end{equation}
This is equal to zero if and only if $i \neq r$, and hence
the column matrices $B C_i$ are linearly independent over $\R$, and so $B$ is invertible.

Now write $C'$ for the vector consisting of $k$ ones, so that the $r$th row of
$B C'$ is equal to $a(n_r)$. We chose $n_r$ so that there are no cubic fields of discriminant
$n_r$ for any $r$, so that the Shintani zeta function counts only the quadratic field of discriminant $n_r$.
By \eqref{eqn_def_an}, $a(n_r) = 0$ for each $n_r$, and hence $B C' = 0$.
But this contradicts the invertibility of $B$; therefore the representation \eqref{eqn_def_an} cannot exist.
\\
\\
{\itshape Positive discriminants.}
To prove our result for the positive Shintani zeta function, we again
define $a(n)$ as in \eqref{eqn_def_an}, again
subtracting $ \zeta(s)$. We choose positive fundamental discriminants $n_i \equiv 21 \ (\textmod \ 36)$ as before,
with no other prime factors in common, and
such that there are no cubic fields of discriminant $- n_i/3$. It follows (by \eqref{eqn_cm_re}, or more simply by
the Scholz reflection principle) that there are also no cubic fields of discriminant $n_i$.
We choose primes $p_i$ as
before, only this time we require that $\omega_{n_i}(p_i) = 2$ so that there will again be a cubic field of discriminant
$n_i p_i^2$. This condition holds if $p_i$ splits completely in $E(\sqrt{-3 n_i})$, where $E$ is as in 
\eqref{eqn_cm_re}. Roughly speaking, we can find such a $p_i$ because this condition is non-abelian.

More precisely, we want to choose $p_i$ so that $\big( \frac{-3 n_i}{p_i} \big) = 1$ (implied if
$p_i$ splits completely in $E(\sqrt{-3 n_i})$), that $\big( \frac{n_i}{p_i} \big) = 1$ (implied by the previous statement
when
$p \equiv 1 \ (\textmod \ 3)$ and therefore $\big( \frac{-3}{p_i} \big) = 1$), and that 
$\big( \frac{n_j}{p_i}) = -1$ for $j \neq i$. As the extensions $E(\sqrt{-3 n_i})$, $\Q(\sqrt{-3})$, and $\Q(\sqrt{n_j})$ ($j \neq i$) are
all Galois and disjoint over $\Q$, the Galois group of their compositum is the direct product of their
Galois groups, and thus the Chebotarev density theorem implies that these splitting conditions may all
be simultaneously satisfied.

Finally, note that the quadratic fields and rings may be handled exactly as in
the negative discriminant case, and so the remainder
of the proof is unchanged.

\end{proof}

\subsection{Generalizations of our results}
Our proof generalizes with minimal modification to some additional Dirichlet series related to Shintani's zeta functions. We now describe examples of such
Dirichlet series and the modifications required.
\\
\\
{\itshape The Dirichlet series for cubic fields.} A Dirichlet series related to the Shintani zeta function is that
simply counting cubic fields:
\begin{equation}
F^{\pm}(s) := \sum_{\substack{[K : \Q] = 3 \\ \pm \Disc(K) > 0}} |\Disc(K)|^{-s}.
\end{equation}
This does not seem to have any nice analytic properties (such as meromorphic continuation to $\mathbb{C}$),
but we cannot rule such properties out. Needless to say, if $F^{\pm}(s)$ had such analytic properties then
this would have interesting consequences for the distribution of cubic fields. For example, this would presumably
yield another proof of the Davenport-Heilbronn theorem, possibly with better error terms than are currently known.

Our argument also proves that neither series $F^{\pm}(s)$ can be represented as a finite sum of Euler
products. The proof is exactly as before, only easier because there are no quadratic fields to count.
We write
$\sum_n a(n) n^{-s} = F^{\pm}(s)$ instead of subtracting $\zeta(s)$ as in \eqref{eqn_def_an}; we thus
have $a(n_i) = a(n_i p_j^2) = 0$ for each $n_i$ when $j \neq i$, and $a(n _i p_i^2) > 0$, and hence the conclusion of the argument works.

The same argument also works if each field is counted with the weight $\frac{1}{|\Aut(K)|}$; indeed, our proof does not
see any cyclic cubic fields.
\\
\\
{\itshape Linear combinations.} Suppose that $\xi(s) = C_1 \xi^+(s) + C_2 \xi^-(s)$ is a linear combination of the usual
Shintani zeta functions, with $C_1 \cdot C_2 \neq 0$. Of particular interest are the cases $C_1 = \sqrt{3}, \ C_2 = \pm 1$,
as the resulting zeta functions have particularly nice functional equations \cite{DW2, N, O}.
We can prove once again that $\xi(s)$ is not a finite sum of Euler products:
We choose negative fundamental discriminants
$n_i$ as above, and the same proof works exactly. (Instead of subtracting $\zeta(s)$ we instead subtract
$C_2 \zeta(s)$.) The point is that all of the $n$ considered in the proof are $\equiv 3 \ (\textmod \ 4)$,
corresponding to fields and nonmaximal rings of discriminant $-n$, 
and so no positive fundamental discriminants enter our calculations. We could similarly work only with positive discriminants
and ignore the negative ones.

\end{document}